\documentclass[11pt]{amsart}

\usepackage{amsfonts,graphics,amsmath,amsthm,amsfonts,amscd,amssymb,amsmath,latexsym,multicol,
 mathrsfs}
\usepackage{epsfig,url}
\usepackage{flafter}
\usepackage{fancyhdr}
\usepackage{hyperref}
\usepackage{amsmath,bm}
\hypersetup{colorlinks=true,colorlinks=true,linkcolor=blue,urlcolor=blue, citecolor=blue}
\makeatletter

\def\jobis#1{FF\fi
  \def\predicate{#1}%
  \edef\predicate{\expandafter\strip@prefix\meaning\predicate}%
  \edef\job{\jobname}%
  \ifx\job\predicate
}

\makeatother

\if\jobis{proposal}%
\else
\fi

 \usepackage[matrix, arrow]{xy}

\DeclareMathOperator{\cent}{center}

\DeclareMathOperator{\mld}{mld}
\DeclareMathOperator{\mult}{mult}

\DeclareMathOperator{\lct}{lct}

 \numberwithin{equation}{section}
 \numberwithin{footnote}{section}

 \newtheorem{thm}{Theorem}[section]

 \newtheorem{lem}[thm]{Lemma}
 
 \newtheorem{conj}[thm]{Conjecture}

{%_%_%_%_% upright style; roman (non-italic) text
    \newtheoremstyle{upright}%
        {8pt plus2pt minus4pt}%
        {8pt plus2pt minus4pt}%
        {\upshape}%
        {}%
        {\bfseries\scshape}%
        {}%
        {1em}%
        {}%
\theoremstyle{upright}

 \newtheorem{defn}[thm]{Definition}
 
 \newtheorem{exa}[thm]{Example}

 \newtheorem{rem}[thm]{Remark}

}

 \newcommand{\Q}{\mathbb Q}
 \newcommand{\R}{\mathbb R}
 \newcommand{\Z}{\mathbb Z}
 \newcommand{\bM}{{\bf{M}}}
 \newcommand{\bN}{{\bf{N}}}
 \newcommand{\bB}{{\bf{B}}}
 \newcommand{\bD}{{\bf{D}}}
 \newcommand{\bL}{{\bf{L}}}
 \newcommand{\bb}{\bm{b}}
 
\title{Effective bound for singularities on toric fibrations}
\author{Bingyi Chen}
\address{Bingyi Chen, Department of Mathematics,
Sun Yat-sen University,
Guangzhou, 510275, P. R. China.}
\email{chenby253@mail.sysu.edu.cn, chenby16@tsinghua.org.cn}
\begin{document}

\begin{abstract}
It was conjectured by M\textsuperscript{c}Kernan and Shokurov that
for any Fano contraction $f:X \to Z$ of relative dimension $r$ with $X$ being $\epsilon$-lc, there is a positive $\delta$ depending only on $r,\epsilon$ such that $Z$ is $\delta$-lc and the multiplicity of the fiber of $f$ over a codimension one point of $Z$ is bounded from above by $1/\delta$. Recently, this conjecture was confirmed by Birkar \cite{Bi23}.
In this paper, we give an explicit value for $\delta$ in terms of $\epsilon,r$ in the toric case, which belongs to $O(\epsilon^{2^r})$ as $\epsilon\rightarrow 0$. The order $O(\epsilon^{2^r})$ is optimal in some sense. 
\end{abstract}

\maketitle
%\tableofcontents

%%%%%%%%%%%%%%%%%%%%%%%%
%%%%%%%%%%%%%%%%%%%%%%%%%%

\section{Introduction}

We work over an algebraically closed field $k$ of characteristic zero. Given a contraction $f:X\rightarrow Z$, i.e. a projective morphism such that $f_* \mathcal O_X=\mathcal O_Z$, a fundamental problem is to relate the singularities on $X$ and those on $Z$. This problem is important as it appears frequently in inductive arguments.
Assume that $f$ is a Fano contraction, M\textsuperscript{c}Kernan conjectured that in this case the singularities on $Z$ are bounded in terms of those on $X$.

\begin{conj}[M\textsuperscript{c}Kernan]\label{conj:mckernan}
Fix a positive integer $r$ and a real number $0<\epsilon\leq 1$. There exists $\delta>0$ depending only on $r,\epsilon$ and satisfying the following. Assume 
\begin{itemize} 
\item $f:X\rightarrow Z$ is a contraction of relative dimension $r$,
\item $X$ is $\epsilon$-lc,
\item $-K_X$ is ample over $Z$, and 
\item $Z$ is $\Q$-Gorenstein. 
\end{itemize}
Then $Z$ is $\delta$-lc.
\end{conj}

Recently, this conjecture was solved by Birkar \cite{Bi23}. Indeed, he proved a more general conjecture -- Shokurov conjecture (see Conjecture \ref{conj:shokurov} below), which implies M\textsuperscript{c}Kernan conjecture. Another interesting consequence of Shokurov conjecture is that under the setting of Conjecture \ref{conj:mckernan}, the multiplicity of the fiber of $f$ over a codimension one point of $Z$ is bounded above.
For more historical results on these two conjectures, we refer to \cite{MP08,MP09,AB14,Bi16,Bi18,BC21,Bi22,HJL22,Ch22}.

The next problem is to give an explicit value for $\delta$ in terms of $r,\epsilon$. When $r=1$ and $\epsilon=1$, Han, Jiang and Luo \cite{HJL22} showed that the optimal value of $\delta$ is 1/2. When $r=1$, in \cite{Ch22} the author  showed that one can take $\delta=\epsilon^2/2$. The main purpose of this paper is to treat the toric case for arbitrary $r,\epsilon$. Our main result is the following.

\begin{thm}\label{thm1}
Let $r$ be a positive integer and $0<\epsilon\leq 1$ be a real number. Let $f:X\rightarrow Z$ be  a toric contraction of relative dimension r such that $-K_X$ is ample over $Z$ and $X$ is vertical$/Z$ $\epsilon$-lc, i.e. $a(E,X,0)\geq \epsilon$ for any prime divisor $E$ over $X$ with $f(\cent_X E)\neq Z$. Let 
\begin{equation}\label{formula}
\delta=\delta(r,\epsilon)= \frac{\epsilon^{2^r}}{2^{2^r-1}\prod\limits_{i=1}^r i^{2^i}}.
\end{equation}
Then

(1) if $Z$ is $\Q$-Gorenstein, then $Z$ is $\delta$-lc;

(2) for any codimension one point $z$ of $Z$, the multiplicity of each component of $f^*z$ is bounded from above by $1/\delta$.
\end{thm}

\begin{rem}
(1) Comparing with Conjecture \ref{conj:mckernan}, in Theorem \ref{thm1} we require a weaker condition that $X$ is vertical$/Z$ $\epsilon$-lc instead of the original condition that $X$ is $\epsilon$-lc. Note that under the original condition the general fiber $F$ of $f$ is an $\epsilon$-lc Fano variety, so it belongs to a bounded family by \cite{Bi19,Bi21}. However, under the new condition, the general fiber may not belong to a bounded family.

(2) For the first assertion in Theorem \ref{thm1}, when $r=2$, the order $O(\epsilon^4)$ is optimal. Indeed, Alexeev and Borisov
\cite[Theorem 1.5]{AB14} constructed a sequence of toric Fano contractions $X\to Z$ such that $\dim X=4$, $\dim Z=2$, $\mld(X)\to 0$ and $\mld(Z)\approx C\cdot \mld(X)^4$.

(3) For the second assertion in Theorem \ref{thm1}, the order $O(\epsilon^{2^r})$ is optimal by the following example.
\end{rem}

\begin{exa}\label{exanew}
Let $q,r$ be two positive integers. Let $u_{i,q}~(i\in \Z_{>0})$ be a sequence of integers defined recursively as follows:
$$u_{1,q}=q,\quad u_{k+1,q}=u_{k,q}(u_{k,q}+1) \text{ for any } k\in \Z_{>0}.$$
Then $u_{r+1,q}\in O(q^{2^r})$ when $q\to +\infty$.

Let $e_1,\cdots,e_{r+1}$ be the standard basis of $\Z^{r+1}$ and denote $e=\sum_{i=1}^r e_i$. Let 
\begin{align*}
&v_i=(1+u_{i,q})e_1-qe \text{ for $1\leq i\leq r$},\\
&v_{r+1}=-e, \quad v_{r+2}=(u_{r+1,q}-1)e_{r+1}-qe.
\end{align*}
%$$v_i=(1+u_{i,q})e_1-qe~(1\leq i\leq r),~v_{r+1}=-e,~v_{r+2}=(u_{r+1,q}-1)e_{r+1}-qe.$$ 
Let $X$ be the toric variety associated to the fan in $\R^{r+1}$ whose maximal cones are generated by $v_{r+2}$ and subsets of $\{v_1,\cdots,v_{r+1}\}$ of size $r$. The support of the fan of $X$ is $\R^r\times \R_{\geq 0}$. The projection $\Z^{r+1}\to \Z$ onto the $(r+1)$-th coordinate induces a toric morphism $f:X\to Z$, where $Z=\mathbb A^1$ with a distinguished point $o$. Then $f:X\to Z$ is a toric Fano contraction of relative dimension $r$.
%with $-K_X$ being ample over $Z$. 
Moreover, 
$$f^*o=(u_{r+1,q}-1)\cdot D,$$
where $D$ is the toric divisor on $X$ corresponding to the ray $\R_{\geq 0} \cdot v_{r+2}$.

Let $S$ be the lattice simplex in $\R^{r+1}$ with vertices $v_1,\cdots,v_{r+2}$. Let $F$ be the face of $S$ which is the intersection of $S$ and the subspace spanned by $e_1,\cdots,e_r$. Then $X$ is $\frac{1}{q}$-lc if and only if 
$$\text{int}(\frac{1}{q} S)\cap \Z^{r+1}=\emptyset \quad \text{and}  \quad \text{relint}(\frac{1}{q} F)\cap \Z^{r+1}=\{\textbf{0}\}.$$
This condition is satisfied because $S$ is contained in the lattice simplex $S'$ with vertices
$$v_i~(1\leq i\leq r),~u_{r+1,q}e_{r+1}-qe,~-u_{r+1,q}e_{r+1}-qe$$
and $\text{int}(\frac{1}{q}S')\cap \Z^{r+1}=\{\textbf{0}\}$ by \cite[Theorem 1.3]{Zo24}. Therefore $X$ is $\frac{1}{q}$-lc.
\end{exa}

The following is a local and more general version of Theorem \ref{thm1}.

\begin{thm}\label{thm2}
Let $r$ be a positive integer, $0<\epsilon\leq 1$ be a real number and $\delta=\delta(r,\epsilon)$  as in \eqref{formula}. Let $f:X\rightarrow Z$ be a toric contraction of relative dimension $r$ and $z\in Z$ be a codimension $\geq 1$ point. Suppose there is a pair $(X,B)$ on $X$ such that $K_X+B\sim_{\mathbb R} 0/Z$ and $\mld(X/Z\ni z, B)\geq \epsilon$. Then

(1) if $Z$ is $\Q$-Gorenstein, then $\mld(Z\ni z, 0)\geq \delta$;

(2) if the codimension of $z$ in $Z$ is one, then the multiplicity of each component of $f^*z$ is bounded from  above by $1/\delta$.
\end{thm}

Here we denote by $\mld(X/Z\ni z,B)$ (resp. $\mld(Z\ni z,0)$) the infimum of the log discrepancy of $E$ with respect to $(X,B)$ (resp. $(Z,0)$) where $E$ runs over all prime divisors over $X$ (resp. $Z$) whose image on $Z$ is $\overline{z}$ (see Definition \ref{mld}).

%\medskip

\begin{rem}
Notice that the assumption ``$\mld(X/Z\ni z, B)\geq \epsilon$'' is weaker than ``$X$ is $\epsilon$-lc over some neighborhood of $z$'' since the former does not put restriction on the log discrepancy of such prime divisor whose image on $Z$ is not $\overline{z}$ but contains $z$. %Therefore, as a consequence of Theorem \ref{thm2}, we can replace the condition ``$X$ is $\epsilon$-lc'' in Theorem \ref{thm1} by a weaker condition ``$a(E,X,0)\geq \epsilon$ for any prime divisor $E$ over $X$ with $f(\cent_X(E))\neq Z$''.
\end{rem}

As mentioned earlier, Shokurov proposed a more general conjecture which implies M\textsuperscript{c}Kernan conjecture. In order to state Shokurov conjecture, we recall some background on adjunction for fibrations (also called the canonical bundle formula) . Let $f:X\rightarrow Z$ be a contraction between normal varieties. Let $(X,B)$ be a pair which is lc over the generic point of $Z$ and such that $K_X+B\sim_{\R} 0/Z$. By the work of Kawamata \cite{Ka97,Ka98} and Ambro \cite{A99,A05}, we may write
$$K_X+B\sim_{\R} f^*(K_Z+B_Z+M_Z),$$
where $B_Z$ is called the discriminant divisor and $M_Z$ is called the moduli divisor. The discriminant divisor is defined by lc threholds, more precisely, the coefficient of a prime divisor $D$ in $B_Z$ is set to be $1-t$ where $t$ is the largest number such that $(X,B+tf^*D)$ is lc over the generic point of $D$. The moduli divisor is then automatically determined up to $\R$-linear equivalence. 

For each birational model $Z'$ over $Z$, similarly we can define $B_{Z'},M_{Z'}$ so that their pushdowns on $Z$ coincide with $B_Z,M_Z$. This defines a discriminant $\bb$-divisor $\bB$ and a moduli $\bb$-divisor $\bM$ over $Z$. It was shown that $\bM$ is a $\bb$-nef $\bb$-divisor and we hence obtain a generalized pair $(Z,B_Z,\bM)$, which is called the generalized pair given by adjunction for $f:(X,B)\rightarrow Z$. See \S \ref{adjunction} for more details. 
%A projective variety $F$ is said to be of Fano type if it admits a klt pair $(F,\Delta)$ such that $-(K_F+\Delta)$ is ample. 
We are now ready to state Shokurov conjecture.
\begin{conj}[Shokurov]\label{conj:shokurov}
  Fix a positive integer $r$ and a real number $0<\epsilon\leq 1$. There exists $\delta>0$ depending only on $r,\epsilon$ and satisfying the following. Let $(X,B)$ be a pair and $f:X\rightarrow Z$ be a contraction such that
  
  \begin{itemize}
  
  \item $\dim X - \dim Z=r$,
  
  \item $(X,B)$ is $\epsilon$-lc,
  
  \item $K_X+B\sim_{\mathbb R} 0/Z$, and 
  
  \item $X$ is of Fano type over $Z$, equivalently, $-K_X$ is big over $Z$.
  \end{itemize}
  Let $(Z, B_Z,\bM)$ be the generalized pair given by adjunction for $f:(X,B)\rightarrow Z$. Then $(Z, B_Z,\bM)$ is generalized $\delta$-lc.
\end{conj}

\iffalse
\begin{conj}[Shokurov]\label{conj:shokurov}
Fix a positive integer $r$ and a real number $0<\epsilon\leq 1$. There exists $\delta>0$ depending only on $r,\epsilon$ and satisfying the following. Let $(X,B)$ be a pair and $f:X\rightarrow Z$ be a contraction with $z \in Z$ a codimension $\geq 1$ point such that

\begin{itemize}

\item $\dim X - \dim Z=r$,

\item $\mld(X/Z\ni z,B)\geq \epsilon$,

\item $K_X+B\sim_{\mathbb R} 0/Z$, and 

\item $X$ is of Fano type over $Z$.
\end{itemize}
Let $(Z, B_Z,\bM)$ be the generalized pair given by adjunction for $f:(X,B)\rightarrow Z$. Then $$\mld(Z\ni z, B_Z,\bM)\geq \delta.$$
\end{conj}
\fi

As mentioned earlier, Shokurov conjecture was proved by Birkar \cite{Bi23} recently. 
%However, it seems not easy to get an explicit value for $\delta$ from the proof in \cite{Bi23} directly. 
Before this celebrated result, in \cite[Theorem 1.4]{BC21} Birkar and Y. Chen showed a variant of Shokurov conjecture in the toric setting, which says that Shokurov conjecture holds after taking an average with the toric boundary. This is enough for some interesting applications. Building on ideas from their work and combining the main result in \cite{Ch22}, we give an explicit value for $\delta$ in \cite[Theorem 1.4]{BC21} as follows.

\begin{thm}\label{thmtech}
Let $r$ be a positive integer and $0<\epsilon\leq 1$ be a real number.
%Then there exist rational numbers $\alpha,\delta\in (0,1)$ depending only on $d,\epsilon$ satisfying the following. 
Assume
\begin{itemize}
  \item $f\colon X\rightarrow Z$ is a toric contraction of relative dimension $r$ with $z\in Z$ a codimension $\geq 1$ point,
  \item $(X,B)$ is a pair (B is not necessarily toric) such that $\mld(X/Z\ni z,B)\geq \epsilon$, 
  \item $K_X+B\sim_{\R}0/Z$, and
  \item $\Delta$ is the toric boundary divisor of $X$.
\end{itemize}
Let
  $$
  \Gamma^{\alpha}=\alpha B+(1-\alpha)\Delta \quad \text{ where } \alpha=1/r!
  $$
and let $(Z,\Gamma^{\alpha}_Z,\bN^{\alpha})$ be the generalized pair given by adjunction for $f:(X,\Gamma^{\alpha})\rightarrow Z$.
Then 
%one can choose $N_Z^{\alpha}\geq 0$ representing the moduli part such that 
$$\mld(Z\ni z, \Gamma_Z^{\alpha},\bN^{\alpha}) \geq \delta,$$
where $\delta=\delta(r,\epsilon)$  as in \eqref{formula}.
\end{thm}

Theorem \ref{thm1} and \ref{thm2} are consequences of Theorem \ref{thmtech}. Another interesting corollary is the following.

\begin{thm}\label{thm3}
Let $r$ be a positive integer, $0<\epsilon\leq 1$ be a real number and $\delta=\delta(r,\epsilon)$  as in \eqref{formula}. Let $f:X\rightarrow Z$ be a toric contraction of relative dimension $r$ with a toric pair $(X,B)$ and a codimension one point $z\in Z$. Suppose there is an $\R$-divisor $B^+$ (not necessary toric) such that $B^+\geq B$, $K_X+B^+\sim_{\mathbb R} 0/Z$ and $\mld(X/Z\ni z, B^+)\geq \epsilon$. Then $(X,B+\delta f^*\overline{z})$ is lc over some neighborhood of $z\in Z$.
%$$\lct(X/Z\ni z,B; f^*z)\geq \delta,$$ 
\end{thm}

\begin{rem}\label{remA}
After this work was completed, Ambro informed me that he also got some explicit lower bounds in the toric case.  Let $f:X \to Z$ be a toric Fano contraction of relative dimension $r$ with $X$ being $\epsilon$-lc. Let $F$ be the general fiber and let $\gamma$ be the $\alpha$-invariant of $F$. There exists a sharp lower bound for $\gamma$ just in terms of $r$ and $\epsilon$ (cf. \cite{A16}). Ambro got explicit bounds for $\delta$ in terms of $\epsilon,r$ and $\gamma$ in Theorem \ref{thm1}, \ref{thm2} and \ref{thm3}. 
\iffalse
Then:

- In Theorem 1.8 he gets the bound $\delta= 2(\gamma \epsilon)^r/d!$

- In Theorem 1.4 (1) he gets the bound $\delta=1/q^{r+1}$, where $q$ is the round up of $\frac{1+r/\gamma}{\epsilon}$.
\fi
\end{rem}

\subsection*{Idea of the proof of Theorem \ref{thmtech}}
The proof is built on ideas from \cite{BC21} with some modifications. In \cite{BC21}, by running toric minimal model program, they reduced the problem to the case for toric Mori fiber spaces. Then they showed that after taking a finite cover and extracting a toric divisor, a $\Q$-factorial toric Mori fiber space can be factored as the composition of toric contractions of smaller relative dimension. Therefore they can reduce the problem to the case for contractions of relative dimension one. However, after taking a finite cover and extracting a divisor, the pullback of $K_X+B$ may be a sub-pair rather than a pair, so it is necessary to take average $\Gamma^{\alpha}=\alpha B+(1-\alpha)\Delta$ with the toric boundary to make its pullback a pair. To guarantee that the singularities of $(X,\Gamma^{\alpha})$ are not too bad, $\alpha$ can not be too small and hence it is important to control the order $n$ of the finite cover and the log discrepancy of the extracted divisor. They showed the boundedness of the order $n$, however, it seems not easy to give an explicit bound for $n$, as it involves all possibilities of the fans corresponding to 
%$\Q$-factorial 
$\epsilon$-lc toric Fano varieties 
%$F$ 
%with $\rho(F)=1$ and $\dim F=r$ 
up to the action of $GL_r(\Z)$. In this paper, we make some modifications to their method. We factor a toric Mori fiber space after extracting a toric divisor with log discrepancy $\leq r$, without taking a finite cover (see Lemma \ref{diagram}). Recall that in relative dimension one, an explicit value for $\delta$ in Shokurov conjecture was given in \cite{Ch22}.  Combining this result we obtain an explicit value for $\delta$ in \cite[Theorem 1.4]{BC21}. 

%\medskip

\subsection*{Acknowledgements} I would like to thank Caucher Birkar for his valuable comments and constant support. I would also like to thank Yifei Chen and Yu Zou for their helpful comments. I am grateful to Florin Ambro for sharing with me his result (Remark \ref{remA}) and for a lot of useful discussions. 

\section{Preliminaries}

%We work over an algebraically closed field $k$ of characteristic zero: all varieties and schemes are over $k$ unless stated otherwise.

We will freely use the standard notations and definitions in \cite{KM98,BCHM10}. A contraction $f:X\rightarrow Z$ is a projective morphism of varieties with $f_*\mathcal{O}_X=\mathcal{O}_Z$. An extremal contraction is a contraction $f:X\rightarrow Z$  with the relative Picard number $\rho(X/Z)=1$.

\subsection{Fano type varieties} Let $X\to Z$ be a contraction of normal varieties. We say $X$ is of Fano type over $Z$ if there is a klt pair $(X,B)$ on $X$ such that $-(K_X+B)$ is ample over $Z$.

We say $X\to Z$ is a Mori fiber space if $-K_X$ is ample over $Z$ and the relative Picard number $\rho(X/Z)=1$.

\subsection{\emph{b}-divisors}
Let $X$ be a normal variety. A $\bb$-divisor $\textbf{D}$ over $X$ is a collection of $\R$-divisors $\textbf{D}_Y$ for each birational model $Y$ over $X$, such that $\sigma_*\textbf{D}_{Y_1}=\textbf{D}_{Y_2}$ for any birational morphism $\sigma:Y_1\rightarrow Y_2/X$.  %We say $\bD$ is a $\Q$-$\bb$-divisor if $\textbf{D}_Y$ is a $\Q$-divisor for each birational model $Y$ over $X$. 

Let $\bD$ be a $\bb$-divisor over $X$ and $Y_0$ be a birational model over $X$. We say $\bD$ descends on $Y_0$ if $\textbf{D}_{Y_0}$ is an $\R$-Cartier $\R$-divisor and $\textbf{D}_{Y}=\sigma^* \textbf{D}_{Y_0}$ for any birational morphism $\sigma:Y\rightarrow Y_0/X$.

Let $X\rightarrow U$ be a projective morphism. We say that a $\bb$-divisor $\textbf{D}$ over $X$ is $\bb$-nef$/U$ (resp. $\bb$-semiample$/U$) if $\bD$ descends on some birational model $Y_0$ and $\textbf{D}_{Y_0}$ is nef$/U$ (resp. semiample$/U$). 
%We say $\bD$ is $\bb$-NQC$/U$ if $\bD$ can be written as an $\R_{\geq 0}$-linear combination of $\bb$-nef$/U$ $\Q$-$\bb$-divisors.

We denote by $\textbf{0}$ the $\bb$-divisor $\bD$ such that $\textbf{D}_Y=0$ for each birational model $Y$ over $X$.

\subsection{Generalized pairs}
We will follow the original definitions in \cite{BZ16} and adopt the notations in \cite{HL21}. 
%Notice that there is a small difference in this paper: all generalized (sub-)pairs are assumed to be NQC unless stated otherwise.

\begin{defn} 
A generalized sub-pair (g-sub-pair for short) $(X,B,\bM)/U$ consists of a normal variety $X$ associated with a projective morphism $X\rightarrow U$, an $\R$-divisor $B$ on $X$, and a  $\bb$-nef$/U$ $\bb$-divisor $\bM$ over $X$.

A g-sub-pair $(X,B,\bM)/U$ is called a sub-pair if $\bM=\textbf{0}$. In this case we denote it by $(X,B)/U$ or $(X,B)$.

A g-sub-pair $(X,B,\bM)/U$ is called a generalized pair (g-pair for short) if $B\geq 0$. A sub-pair $(X,B)$ is called a pair if $B\geq 0$.

We may drop $U$ when we emphasize the structures of $(X,B,\bM)$ that are independent of the choice of $U$, for example, the singularities of $(X,B,\bM)$.
% See Definition \ref{sing} below. 
\end{defn}

\begin{defn} \label{sing}
Let $(X,B,\bM)/U$ be a g-(sub-)pair and $E$ be a prime divisor over $X$, i.e. a prime divisor on a normal variety $Y$ with a birational morphism $\pi:Y\rightarrow X$. The center of $E$ on $X$ is defined as the image of $E$ on $X$ under the morphism $\pi$ and it is denoted by $\cent_X E$. Write 
$$K_Y+B_Y+\bM_Y:=\pi^*(K_X+B+\bM_X).$$
Then the log discrepancy of $E$ with respect to $(X,B,\bM)$ is defined as $1-\mult_E B_Y$ and it is denoted by $a(E,X,B,\bM)$, where $\mult_E B_Y$ is the coefficient of $E$ in $B_Y$.
\end{defn}
\begin{defn}\label{mld}
Let $(X,B,\bM)/U$ be a g-(sub-)-pair, $f:X\rightarrow Z/U$ be a projective morphism and $z\in Z$ be a (not necessary closed) point. The minimal log discrepancy of $(X,B,\bM)$ over $z$ is defined as
\begin{align*}
\mld(X/Z\ni z,B,\bM):=\inf\{a(E,X,B,\bM) \mid ~ &\text{$E$ is a prime divisor over $X$}\\
&\text{with $f(\cent_X(E))=\overline{z}$}\}.
\end{align*}
In the case that $Z=X$, $z=x$ and $f$ is the identity morphism, we will use  $\mld(X\ni x,B,\bM)$ instead of $\mld(X/Z\ni z,B,\bM)$.
\end{defn}

\begin{defn}
A g-(sub-)pair $(X,B,\bM)$ is said to be (sub-)$\epsilon$-glc (resp. (sub-)$\epsilon$-gklt, (sub-)glc, (sub-)gklt) if $\mld(X\ni x,B,\bM)\geq \epsilon$ (resp. $>\epsilon$,~$\geq 0$,~$> 0$) for any codimension $\geq 1$ point $x\in X$. 

If $\bM=0$ and $(X,B,\bM)$ is (sub-)$\epsilon$-glc (resp. (sub-)$\epsilon$-gklt, (sub-)glc, (sub-)gklt), we say that $(X,B)$ is (sub-)$\epsilon$-lc (resp. (sub-)$\epsilon$-klt, (sub-)lc, (sub-)klt). In the case when $B=0$, we also say $X$ is $\epsilon$-lc (resp. $\epsilon$-klt, lc, klt). 
\end{defn}

\begin{defn}
Let $(X,B,\bM)/U$ be a g-(sub-)pair and $D$ be an effective $\R$-Cartier $\R$-divisor on $X$. The lc threshold of  $D$ with respect to  $(X,B,\bM)$ is defined to be
$$\lct(X,B,\bM;D):=\sup\{t\geq 0\mid (X,B+tD,\bM) \text{ is (sub-)glc}\}.$$
%If $z\in Z$ is a codimension 1 point, then $\overline{z}$ is a Cartier divisor on some neighborhood $U$ of $z\in Z$. We define$$\lct(X/Z\ni z,B;\pi^*z):=\sup\{t\geq 0\mid (X/Z\ni z,B+t\pi^*z) \text{ is lc over $U$}\}.$$ This definition is independent of the choice of $U$.
\end{defn}

\begin{defn}\label{bettersing}
Let $(X,B,\bM)/U$ and $(X,\Gamma,\bN)/U$ be two g-(sub-)pairs. We say $(X,B,\bM)$ has better singularities than $(X,\Gamma,\bN)$ if 
$$a(E,X,B,\bM)\geq a(E,X,\Gamma,\bN)$$ for any prime divisor $E$ over $X$.
\end{defn}

\begin{lem}\label{lalala}
Let $(X,B,\bM)/U$ be a g-(sub-)-pair, $f:X\rightarrow Z/U$ be a projective morphism and $z\in Z$ be a (not necessary closed) point.
Then $\mld(X/Z\ni z,B,\bM)\geq 0$ if and only if $(X,B,\bM)$ is (sub-)glc over some neighborhood of $z\in Z$.
\end{lem}
\begin{proof}
This is essentially \cite[Lemma 2.8]{HJL22} where it was stated only for $\bM=0$. By definition, the ``if'' part is obvious. Next we show the ``only if'' part. 

Assume the contrary that $\mld(X/Z\ni z,B,\bM)\geq 0$ but $(X,B,\bM)$ is not (sub-)glc over any neighborhood of $z\in Z$. Then there is a prime divisor $E$ over $X$ such that $z\in f(\cent_X E)$ and $a(E,X,B,\bM)< 0$.
Let $\pi:Y\rightarrow X$ be a resolution with 
$K_Y+B_Y+\bM_Y=\pi^*(K_X+B+\bM_X)$
such that 
\begin{itemize}
	\item $\bM$ descends on $Y$,
	\item $E$ is a prime divisor on $Y$,
	\item $\overline{\pi^{-1}f^{-1}(z)}$ is a divisor on $Y$, say $F$, and 
	\item $E+F$ is a simple normal crossing divisor on $Y$. 
\end{itemize}
We can find an irreducible component $D$ of $F$ such that $f(\pi(D\cap E))=\overline{z}$ (indeed, since $z\in f(\pi(E))$, there is a point $\eta\in E$ such that $f(\pi(\eta))=z$, then we take $D$ to be a component of $F$ which contains $\eta$). 

Denote $d=\mult_D B_Y$ and $e=\mult_E B_Y>1$. Blowing up $D\cap E$, we get a new resolution $\pi':Y'\rightarrow X$ with $K_{Y'}+B_{Y'}+\bM_{Y'}=\pi'^*(K_X+B+\bM_X)$. Denote by $D'$ the exceptional$/Y$ divisor on $Y'$ and by $E'$ the birational transformation of $E$ on $Y'$. By construction, we have $f(\pi'(D'))=\overline{z}$, $D'$ meets $E'$ transversely, $f(\pi'(D'\cap E'))=\overline{z}$ and $\mult_{D'} B_{Y'}\geq d+e-1>d$.

So, by successively blowing up, we eventually obtain a prime divisor $\widetilde{D}$ over $X$ such that $f(\cent_X \widetilde{D})=\overline{z}$ and $a(\widetilde{D},X,B,\bM)<0$, which contradicts that $\mld(X/Z\ni z,B,\bM)\geq 0$.
\end{proof}

%\subsection{Fano type varieties} Let $\pi: X\rightarrow Z$ be a projective morphism of varieties. We say $X$ is of Fano type over $Z$ if there is an $R$-divisor $C$ such that $(X,C)$ is a klt pair and $-(K_X+C)$ is ample over $Z$. In the case that $Z$ is a point, we call $X$ a Fano type variety. The unique 1-dimensional Fano type variety is $\PP^1$. Note that $X$ is of Fano type over $Z$ is equivalent to that the general fibers of $\pi$ are Fano type varieties. 

\subsection{Adjunction for generalized fibrations}\label{adjunction}

\begin{defn}
Let $f:X\rightarrow Z$ be a contraction between normal varieties over $U$ with $\dim Z>0$. Let $(X,B,\bM)/U$ be a g-pair which is glc over the generic point of $Z$ and such that $K_X+B+\bM_X\sim_{\R} 0/Z$. Then  $K_X+B+\bM_X\sim_{\R} f^* L$ for some $\R$-Cartier $\R$-divisor $L$ on $Z$. 

For any prime divisor $D$ on $Z$, let $t_D$ be the lc threshold of $f^*D$ with respect to $(X,B,\bM)$ over the generic point of $D$. This make sense even if $D$ is not $\Q$-Cartier because we only need the pullback of $D$ over the generic point of $D$ where $Z$ is smooth.
We set $B_Z=\sum_D (1-t_D)D$ where $D$ runs over all prime divisors on $Z$
%Then $B_Z$ is an effective divisor.
and set $M_Z=L-K_Z-B_Z$. 
The former is called the discriminant divisor and the latter is called the moduli divisor.

Let $\sigma: Z'\rightarrow Z$ be a birational morphism from a normal variety $Z'$ and let $X'$ be the resolution of the main component of $X\times_Z Z'$ with induced morphism $\tau:X'\rightarrow X$ and $f':X'\rightarrow Z'$. Write $K_{X'}+B'+\bM_{X'}=\tau^*(K_X+B+\bM_X)$, then $K_{X'}+B'+\bM_{X'}\sim_{\R} f'^* \sigma^* L$. Similarly we can define the discriminant divisor $B_{Z'}$ and the moduli divisor $M_{Z'}$ for the contraction $(X',B',\bM)\rightarrow Z'$. One can check that $\sigma_* B_{Z'}=B_Z$ and
$\sigma_* M_{Z'}=M_Z$. Hence there exist $\bb$-divisors $\bB^Z,\textbf{M}^Z$ such that $\bB^Z_{Z'}=B_{Z'}$ and $\textbf{M}^Z_{Z'}=M_{Z'}$ for any birational model $Z'$ over $Z$, which are called the discriminant b-divisor and the moduli b-divisor respectively. By construction, we have
$$K_X+B+\bM_X\sim_{\R} f^*(K_Z+B_Z+\bM^Z_Z).$$

It was shown that $\bM^Z$ is a $\bb$-nef$/U$ $\bb$-divisor over $Z$ (see \cite[Theorem 11.4.4]{CHLX23}). Hence we can regard $(Z,B_Z,\bM^Z)/U$ as a g-pair.  We call $(Z,B_Z,\bM^Z)/U$ the g-pair given by adjunction for $f:(X,B,\bM)\rightarrow Z$. In the case that $(X,B,\bM)$ is glc, $(Z,B_Z,\bM^Z)$ is also a glc g-pair.

For more details about adjunction for generalized fibrations, we refer the readers to \cite{Fi20}, \cite{JLX22} and \cite[\S 11.4]{CHLX23}.

\end{defn}
\iffalse
\begin{lem}
Let $f:X\to Z$ be a contraction of relative dimension one between normal varieties over $U$ with $\dim Z>0$. Let $(X,B,\bM)$ be a glc g-pair on $X$ such that $K_X+B+\bM\sim_{\R} 0/Z$ and $\bM$ is $\bb$-semiample$/U$. Then the moduli $\bb$-divisor ${\bf M}^Z$ over $Z$ is $\bb$-semiample$/U$.
\end{lem}
\begin{proof}
{\cite[Theorem 8.1]{PS09}},{\cite[Theorem 2.14]{HJL22}}
\end{proof}
\fi

\iffalse

\begin{thm}[Theorem 1.4]
Let $f:X\rightarrow Z$ be a contraction, $(X,B)$ a pair and $z\in Z$ a codimension $\geq 1$ point such that

\begin{itemize}

\item $\dim X - \dim Z=1$,

\item $K_X+B\sim_{\mathbb R} 0/Z$, 

\item $\mld(X/Z\ni z,B)\geq \epsilon$, where $0<\epsilon\leq 1$, and

\item $-K_X$ is big$/Z$.
\end{itemize}
Let $(Z,B_Z,\bM)$ be the g-pair on $Z$ induced by the canonical bundle formula 
  $$K_X+B\sim_{\R} f^*(K_Z+B_Z+\bM_Z).$$
Then
\begin{align*}
\mld(Z\ni z, B_Z,\bM) \geq \epsilon^2/4.
\end{align*}
\end{thm}

\fi

\begin{lem}\label{trans}\cite[Lemma 2.1]{BC21}
Assume that
\begin{itemize}
  \item $(X,B,\bM)/U$ is a g-pair which is glc over the generic point of $Z$,
  \item $X \xrightarrow{g} Y\xrightarrow{h} Z$ are contractions of normal varieties$/U$ with $\dim Z>0$, and
  \item $K_X+B+\bM_X \sim_{\mathbb{R}}0/Z$.
\end{itemize}
Let $(Y,B_Y,\bM^Y)/U$ be the g-pair given by adjunction for $g:(X,B,\bM)\rightarrow Y$ and let $(Z,B_Z,\bM^Z)/U$ be the g-pair given by adjunction for $h\circ g:(X,B,\bM)\rightarrow Z$. Then $(Z,B_Z,\bM^Z)/U$ is also the g-pair given by adjunction for $h:(Y,B_Y,\bM^Y)\rightarrow Z$.
\end{lem}

\begin{lem}\label{betterbase}
Let $f:X\rightarrow Z$ be a contraction between normal varieties over $U$. Let $(X,B,\bM)/U$ and $(X,\Gamma,\bN)/U$ be two g-pairs on $X$ which are glc over the generic point of $Z$. Assume that $K_X+B+\bM_X\sim_{\R } 0/Z$ and $K_X+\Gamma+\bN_X\sim_{\R } 0/Z$. Let $(Z,B_Z,\bM^Z)/U$ and $(Z,\Gamma_Z,\bN^Z)/U$ be the g-pairs given by adjunction for $(X,B,\bM)$ and $(X,\Gamma,\bN)$ over $Z$ respectively. Suppose that $(X,B,\bM)$ has better singularities than $(X,\Gamma,\bN)$, then $(Z,B_Z,\bM^Z)$ has better singularities than $(Z,\Gamma_Z,\bN^Z)$ (see Definition \ref{bettersing} for ``better singularities'').
\end{lem}
\begin{proof}
Let $D$ be a prime divisor on some high resolution $Z'\rightarrow Z$ . Let $\pi:X'\rightarrow X$ be a high enough resolution such that the induced map $f':X'\dashrightarrow Z'$ is a morphism. Write $K_{X'}+B'+\bM_{X'}$ (resp. $K_{X'}+\Gamma'+\bN_{X'}$) for the pullback of $K_X+B+\bM_{X}$ (resp. $K_X+\Gamma+\bN_{X}$). Denote by $t$ (resp. $s$) the lc threshold of $f'^*D$ with respect to $(X',B',\bM)$ (resp. $(X',\Gamma',\bN)$) over the generic point of $D$. It suffices to show $t\geq s$.

By construction, $(X',\Gamma'+sf'^*D, \bN)$ is sub-glc over the generic point of $D$. Since $(X,B,\bM)$ has better singularities than $(X,\Gamma,\bN)$, $(X',B'+sf'^*D,\bM)$ also has better singularities than $(X',\Gamma'+sf'^*D, \bN)$ and it hence is sub-glc over the generic point of $D$. Therefore $t\geq s$.
\end{proof}

\begin{lem}\label{better}
Let $f:X\rightarrow Z$ be a contraction of normal varieties over $U$. Let $(X,B,\bM)/U$ and $(X,\Gamma,\bN)/U$ be two g-pairs on $X$ which are glc over the generic point of $Z$. Assume that $K_X+B+\bM_X\sim_{\R } 0/Z$ and $K_X+\Gamma+\bN_X\sim_{\R } 0/Z$. Let $0\leq \alpha\leq 1$ be a real number and let
$$\Omega= \alpha B+(1-\alpha) \Gamma ~\text{ and }~ \bL=\alpha \bM+(1-\alpha) \bN.$$
Let $(Z,B_Z,\bM^Z)/U$, $(Z,\Gamma_Z,\bN^Z)/U$ and $(Z,\Omega_Z,\bL^Z)/U$ be the g-pairs given by adjunction for $(X,B,\bM)$, $(X,\Gamma,\bN)$ and $(X,\Omega,\bL)$ over $Z$ respectively. Then $(Z,\Omega_Z,\bL^Z)$ has better singularities than 
\begin{align}\label{pair}
(Z,\alpha B_Z+(1-\alpha)\Gamma_Z, \alpha \bM^Z+(1-\alpha)\bN^Z).
\end{align}
See Definition \ref{bettersing} for ``better singularities''.
\end{lem}
\begin{proof}
Let $Z'\rightarrow Z$ be any resolution and $D$ be a prime divisor on $Z'$. Take a high enough resolution $X'\rightarrow X$ such that the induced map $h':X'\dashrightarrow Z'$ is a morphism. Let $t$ (resp. $s$) be the lc threshold of $h'^*D$ with respect to $(X',B',\bM)$ (resp. $(X',\Gamma',\bN)$) over the generic point of $D$ where $K_{X'}+B'+\bM_{X'}$ (resp. $K_{X'}+\Gamma'+\bN_{X'}$) is the pullback of $K_X+B+\bM_X$ (resp. $K_{X}+\Gamma+\bN_{X}$). By definition, the coefficient of $D$ in $B_{Z'}$ (resp. $\Gamma_{Z'}$) is $1-t$ (resp. $1-s$) where $K_{Z'}+B_{Z'}+\bM^{Z}_{Z'}$ (resp. $K_{Z'}+\Gamma_{Z'}+\bN_{Z'}^Z$) is the pullback of $K_Z+B_Z+\bM_Z^Z$ (resp. $K_{Z}+\Gamma_Z+\bM_{Z}^Z$). Hence $a(D,Z,B_Z,\bM^Z)=t$ and $a(D,Z,\Gamma_Z,\bN^Z)=s$. So the log discrepancy of $D$ with respect to the g-pair \eqref{pair} is $\alpha t+(1-\alpha)s$.

Now 
$$(X',\alpha B'+(1-\alpha)\Gamma'+\alpha t h'^*D +(1-\alpha) s h'^*D, \alpha \bM+(1-\alpha)\bN)$$
is glc over the generic point of $D$. Assuming that $K_{X'}+\Omega'+\bL_{X'}$ is the pullback of $K_X+\Omega+\bL_X$, we have $\Omega'=\alpha B'+(1-\alpha)\Gamma'$ and $\bL=\alpha \bM+(1-\alpha) \bN$. Hence the lc threshold of $h'^*D$ with respect to $(X',\Omega',\bL)$ over the generic point of $D$ is as least $\alpha t+(1-\alpha)s$. By definition, the coefficient of $D$ in $\Omega_{Z'}$ is at most  $1-\alpha t-(1-\alpha)s$, where $K_{Z'}+\Omega_{Z'}+\bL^{Z}_{Z'}$ is the pullback of $K_Z+\Omega_Z+\bL^Z_Z$. So 
$$a(D,Z,\Omega_Z,\bL^Z)\geq \alpha t+(1-\alpha)s$$
and the proof is completed.
\end{proof}

\subsection{Toric varieties and toric morphisms} We refer to \cite{Fu93}, \cite{Od88} or \cite{CLS11} for the general theory of toric varieties. Here we collect some definitions and facts on toric varieties. All toric varieties in this paper are assumed to be normal.

A toric variety $X$ is given by a pair $(N_X,\Sigma_X)$, where $N_X$ is a lattice and $\Sigma_X$ is a rational polyhedral fan in $N_X\otimes \mathbb{R}$. A toric morphism between toric varieties $X$ and $Y$ is given by a $\Z$-linear map $\phi:N_X\rightarrow N_Y$ which is compatible with the fan $\Sigma_X$ and $\Sigma_Y$, that is to say, for any cone $\sigma_1\in \Sigma_X$, there is a cone $\sigma_2 \in \Sigma_Y$ such that $\phi_{\R}(\sigma_1)\subset \sigma_2$, where $\phi_{\mathbb{R}}:N_X\otimes \mathbb{R}\rightarrow N_Y\otimes \mathbb{R}$ is the extension of $\phi$. 

A toric divisor on a toric variety $X$ is a divisor which is invariant under the torus action. We say a pair $(X,B)$ is a toric pair if $X$ is a toric variety and $B$ is a toric $\R$-divisor.

There is a one-to-one correspondence between the cones $\sigma$ in $\Sigma_X$ and the torus orbits $O(\sigma)$ in $X$. The dimension of the cone $\sigma$  is equal to the  codimension of the orbit $O(\sigma)$ in $X$.
In particular, a one-dimensional cone $\sigma$, called a ray,  corresponds to a toric prime divisor $\overline{O(\sigma)}$.

If $\Delta$ is the toric boundary divisor of a toric variety $X$, that is, $\Delta$ is the sum of all the toric prime divisors on $X$, then $(X,\Delta)$ is lc and $K_X+\Delta\sim 0$. Moreover, $a(D,X,\Delta)=0$ for any toric prime divisor $D$ over $X$.

A toric variety $X$ is $\mathbb{Q}$-factorial if and only if the fan $\Sigma_X$ is simplicial, that is, every cone in $\Sigma_X$ is generated by a set of $\R$-linear independent vectors. 

If a toric morphism $f:X\rightarrow Y$ given by $\phi:N_X\rightarrow N_Y$ is a contraction, then $\phi$ is surjective.

If $f:X\rightarrow Z$ is a toric contraction, then $X$ is of Fano type over $Z$.

Every toric varieties is a Mori dream space, that is to say, if $X\to Z$ is a toric contraction,
then we can run a minimal model program (MMP, for short) on any $\R$-Cartier $\R$-divisor $D$ relatively over $Z$
and it terminates with
either a $D$-negative fibre space or a $D$-minimal model. Moreover, all the steps of the MMP are toric. See 
\cite[Chapter 14]{Ma02} for the $\Q$-factorial case.

\begin{lem}\cite[p.133]{CLS11}\label{contraction-over-torus}
Let $X,Z$ be two toric varieties given by $(N_X,\Sigma_X)$, $(N_Z,\Sigma_Z)$ respectively and $f:X\to Z$ be a toric morphism given by a surjective $\Z$-linear map $\phi:N_X\to N_Z$. Let $F$ be a toric varieties given by $(N_0,\Sigma_0)$ where $N_0=\ker(\phi)$ and 
$$\Sigma_0=\{\sigma\in \Sigma_X\mid \sigma\subset (N_0)_{\R}\}$$
is a sub-fan of $\Sigma_X$. Then $f^{-1}(T_Z)\simeq F\times T_Z$, where $T_Z$  is the torus in $Z$.
\end{lem}

We also need the following lemma in \cite{BC21} regarding adjunction for toric pairs.
\begin{lem}\label{toricadjunction}\cite[Lemma 2.11]{BC21}
Let $f:X\rightarrow Z$ be a toric contraction between toric varieties and $\Delta,\Delta_Z$ be the toric boundary divisors of $X,Z$ respectively. Then $(Z,\Delta_Z,{\bf 0})$ is the g-pair given by adjunction for $f:(X,\Delta)\rightarrow Z$.
\end{lem}

%%%%%%%%%%%%%%%%%%%%%%%%%

\section{Proofs of main results}

In this section, we will prove  a more general form of Theorem \ref{thmtech} for generalized pairs as follows.
\begin{thm}\label{generalized}
Let $r$ be a positive integer and $0<\epsilon\leq 1$ be a real number.
%Then there exist rational numbers $\alpha,\delta\in (0,1)$ depending only on $d,\epsilon$ satisfying the following. 
Assume
\begin{itemize}
  \item[(a)]  $f\colon X\rightarrow Z$ is a toric contraction of relative dimension $r$ between toric varieties over $U$ with a codimension $\geq 1$ point $z\in Z$,
  \item[(b)]  $(X,B,\bM)/U$ is a g-pair (not necessarily toric) with $\mld(X/Z\ni z,B,\bM) \geq \epsilon$,
  \item[(c)]  $K_X+B+\bM_X\sim_{\R}0/Z$, and
  \item[(d)]  $\Delta$ is the toric boundary divisor of $X$.
\end{itemize}
Let
  $$
  \Gamma^{\alpha}=\alpha B+(1-\alpha)\Delta ~ \text{ and } ~\bN^{\alpha}=\alpha \bM \quad \text{ where } \alpha=1/r!
  $$
and let $(Z, \Gamma^{\alpha}_Z,\bN^{\alpha,Z})/U$ be the g-pair on $Z$ given by adjunction for $f:(X,\Gamma^{\alpha},\bN^{\alpha})\rightarrow Z$. Then 
%one can choose $N_Z^{\alpha}\geq 0$ representing the moduli part such that 
$$\mld(Z\ni z, \Gamma^{\alpha}_Z,\bN^{\alpha,Z}) \geq \delta,$$
where $\delta=\delta(r,\epsilon)$  as in \eqref{formula}.
\end{thm}

We start with showing a generalized version of \cite[Theorem 1.4]{Ch22}, i.e. showing that one can take $\delta=\epsilon^2/2$ in a generalized version of Shokurov conjecture in relative dimension one. Its proof is similar to that of \cite[Lemma 3.1]{BC21}.
%We follow the arguments in the proof of \cite[Lemma 3.1]{BC21}
%i.e. giving an explicit value for $\delta$ in a generalized version of Shokurov conjecture in relative dimension one.

\begin{lem}\label{dim1}
Let $f:X\rightarrow Z$ be a contraction between normal varieties over $U$, $(X,B,\bM)/U$ be a g-pair and $z\in Z$ be a codimension $\geq 1$ point such that

\begin{itemize}

\item $\dim X - \dim Z=1$,

\item $K_X+B+\bM_X\sim_{\mathbb R} 0/Z$, 

\item $\mld(X/Z\ni z,B,\bM)\geq \epsilon$, where $0<\epsilon\leq 1$, and

\item $X$ is of Fano type over $Z$.
\end{itemize}
Let $(Z,B_Z,\bM^Z)/U$ be the g-pair given by adjunction for $f:(X,B,\bM)\rightarrow Z$. Then
\begin{align*}
\mld(Z\ni z, B_Z,\bM^Z) \geq \delta(1,\epsilon)=\epsilon^2/2.
\end{align*}
\end{lem}
\begin{proof}
\iffalse
Let $\pi:X'\rightarrow X$ be a high enough resolution such that $\bM$ descend to $X'$. Write $K_{X'}+B'+\bM_{X'}=\pi^*(K_X+B+\bM_X)$.
Since $\bM_{X'}$ is big over $X$, we can write
$$\bM_{X'}\sim_{\R} H'+C'/Z$$
where $H'$ is ample$/X$ and $C'$. For each small $u>0$, we can write
$$\bM_{X'}\sim_{\R} (1-u)\bM_{X'}+uH'+uC'/X.$$
Since $\bM_{X'}$ is nef$/Z$, $(1-u)\bM_{X'}+uH'$ is ample$/X$. Find a general
$$0\leq L'_u\sim_{\R} (1-u)\bM_{X'}+uH'/X.$$
and set
$$\Delta'=B'+L'_u+uC',$$
then $K_{X'}+\Delta'\sim_{\R} 0$ 
\fi

%When $\bM=0$, this lemma is proved  in \cite[Theorem 1.4]{Ch22}.
Since the singularities of $(Z, B_Z,\bM^Z)/U$ are independent of the choice of $U$, we may assume that $U=Z$. Shrinking $Z$ around $z$, by Lemma \ref{lalala} we may assume that $(X,B,\bM)$ is glc. Let $D$ be a prime divisor on some high resolution $Z'\rightarrow Z$ with $\cent_Z D=\overline{z}$. Let $\pi:X'\rightarrow X$ be a high enough resolution such that $\bM$ descends on $X'$ and the induced map $f':X'\dashrightarrow Z'$ is a morphism. Write $K_{X'}+B'+\bM_{X'}$ for the pullback of $K_X+B+\bM_{X}$. Then $(X',B')$ is sub-lc and $\mld(X'/Z\ni z, B')\geq \epsilon$. Denote by $t$ the lc threshold of $f'^*D$ with respect to $(X',B')$ over the generic point of $D$. It suffices to show that $t$ is bounded from below by $\epsilon^2/2$. %Replacing $X'$ by a higher resolution, we may assume that the union of $\Supp B'$, $\Supp f'^*D$, $\Supp C'$ and the exceptional$/X$ divisor is a simple normal crossing divisor on $X'$. So the lc threshold $t$ can simply be calculated from the coefficients of $B'$ and $f'^*D$.

We may assume that $X$ is $\Q$-factorial. Since $X$ is of Fano type over $Z$, $X$ is klt and $-K_X$ is big over $Z$. So we can write 
$$\pi^*(-K_X)\sim_{\Q} H'+C'/Z$$ where $H'$ is ample over $Z$ and $C'\geq 0$. We can also write $\pi^*K_X=K_{X'}+E'$. Then $E'\leq B'$ and $(X',E')$ is sub-klt. For each sufficiently small real number $u>0$, let 
$$B'_u=(1-u)B'+uE',$$
then we have $(X',B'_u)$ is sub-klt and $\mld(X'/Z\ni z, B'_u)\geq \epsilon$. So we can find a general 
$$0\leq L'\sim_{\R} (1-u)\bM_{X'}+uH'/Z$$
(note that $H'$ is ample$/Z$ and $\bM_{X'}$ is nef$/Z$) such that letting
$$\Delta'=B'_u+uC'+L',$$
we have $\mld(X'/Z\in z,\Delta')\geq \epsilon'$ where $\epsilon-\epsilon'>0$ is sufficiently small. Moreover, over $Z$ we have
\begin{align*}
K_{X'}+\Delta'&\sim_{\R} K_{X'}+(1-u)B'+uE'+uC'+(1-u)\bM_{X'}+uH'\\
&=(1-u)(K_{X'}+B'+\bM_{X'})+u(K_{X'}+E')+u(H'+C')\\
&\sim_{\R} (1-u)(K_{X'}+B'+\bM_{X'}) \sim_{\R} 0.
\end{align*}
Therefore, letting $\Delta=\pi_*\Delta'$, we deduce that $K_{X'}+\Delta'$ is the pullback of $K_X+\Delta$. 

Now if we choose $u>0$ to be sufficiently small, the lc threshold $s$ of $f'^* D$ with respect to $(X',\Delta')$ over the generic point of $D$ is sufficiently close to $t$. Applying \cite[Theorem 1.4]{Ch22} to $(X,\Delta)\to Z$, we deduce that $s\geq \epsilon'^2/2$, where $\epsilon-\epsilon'>0$ is sufficiently small. Hence $t\geq \epsilon^2/2$.
\end{proof}

To prove Theorem \ref{generalized}, we need a couple of lemmas.

\begin{lem}\label{complex}
Let  $0<\epsilon\leq 1$ be a real number and $r,s,t$ be positive integers such that $r=s+t$. Suppose Theorem \ref{generalized} holds in relative dimension $s$ and in relative dimension $t$. Keep the assumptions (a),(b),(c) and (d) in Theorem \ref{generalized}. Additionally assume that
$f:X\rightarrow Z$ can be factored as $X\xrightarrow{g} Y \xrightarrow{h} Z$ where $h$ and $g$ are toric contractions of relative dimension $s$ and $t$ respectively. Let
  $$
  \Gamma^{\beta}=\beta B+(1-\beta)\Delta ~ \text{ and } ~\bN^{\beta}=\beta \bM \quad \text{ where } \beta=1/(s!t!)
  $$
and let $(Z, \Gamma^{\beta}_Z,\bN^{\beta,Z})/U$ be the g-pair given by adjunction for $f:(X,\Gamma^{\beta},\bN^{\beta})\rightarrow Z$.
Then  
$$\mld(Z\ni z, \Gamma^{\beta}_Z,\bN^{\beta,Z}) \geq \delta\big(t,\delta(s,\epsilon)\big)= \frac{\epsilon^{2^{s+t}}}{2^{2^{s+t}-1}\prod\limits_{i=1}^s i^{2^{i+t}}\cdot\prod\limits_{i=1}^t i^{2^i}}.$$

\iffalse
(2) Let
  $$
  \Gamma^{\alpha}=\alpha B+(1-\alpha)\Delta \quad \text{ where } \alpha=1/(s+t)!
  $$
and let
  $$K_X+\Gamma^{\alpha}\sim_{\Q} f^*(K_Z+\Gamma_Z^{\alpha}+N_Z^{\alpha})$$
be given by adjunction.
Then  
$$\mld(Z\ni z, \Gamma_Z^{\alpha}+N_Z^{\alpha}) \geq \delta$$
where $\delta=\delta(s+t,\epsilon)$ as in \eqref{formula}.
\fi

\end{lem}
\begin{proof}
By assumption, Theorem \ref{generalized} holds for both $h$ and $g$ in the following sense. Let 
$$\Gamma^{\lambda}=\lambda B+(1-\lambda) \Delta ~\text{ and } ~\bN^{\lambda}=\lambda \bM \quad \text{ where } \lambda=1/s!$$ and let $(Y,\Gamma^{\lambda}_Y,\bN^{\lambda,Y})/U$ be the g-pair given by adjunction for $g:(X,\Gamma^{\lambda},\bN^{\lambda})\rightarrow Y$. 
Then 
$$\mld(Y/Z\ni z,\Gamma^{\lambda}_Y,\bN^{\lambda,Y})\geq \delta(s,\epsilon).$$ On the other hand, let
$$\Omega_{Y}^\gamma=\gamma \Gamma^{\lambda}_Y+(1-\gamma)\Delta_Y ~\text{and}~ \bL^{\gamma,Y}=\gamma \bN^{\lambda,Y}\quad \text{where }\gamma=1/t!$$
and $\Delta_Y$ is the toric boundary divisor of $Y$. Let $(Z,\Omega_{Z}^\gamma,\bL^{\gamma,Z})/U$ be the g-pair given by the adjunction for $h:(Y,\Omega_Y^{\gamma},\bL^{\gamma,Y})\rightarrow Z$. Then 
\begin{align}\label{tt}
\mld(Z\ni z,\Omega_{Z}^\gamma,\bL^{\gamma,Z})\geq \delta\big(t,\delta(s,\epsilon)\big).
\end{align}

Now let
  $$
  \Gamma^{\beta}=\beta B+(1-\beta)\Delta ~ \text{ and } ~\bN^{\beta}=\beta \bM \quad \text{ where } \beta=1/(\lambda\gamma)=1/(s!t!).
  $$
By construction, we have
$$\Gamma^{\beta}=\gamma\Gamma^{\lambda}+(1-\gamma)\Delta ~ \text{ and } ~ \bN^{\beta}=\gamma \bN^{\lambda}.$$
Let $(Y, \Gamma^{\beta}_Y,\bN^{\beta,Y})/U$ be the g-pair given by adjunction for $g:(X,\Gamma^{\beta},\bN^{\beta})\rightarrow Y$.
 and let $(Z, \Gamma^{\beta}_Z,\bN^{\beta,Z})$ be the g-pair given by adjunction for $f:(X,\Gamma^{\beta}+\bN^{\beta})\rightarrow Z$. 
By Lemma \ref{trans},  $(Z, \Gamma^{\beta}_Z,\bN^{\beta,Z})$ is also the g-pair given by adjunction for $h:(Y, \Gamma^{\beta}_Y,\bN^{\beta,Y})\rightarrow Z$.
 %$$K_Y+\Gamma^{\beta}_Y+\bN^{\beta,Y}_Y \sim_{\R} h^*(K_Z+\Gamma^{\beta}_Z+\bN^{\beta,Z}_Z).$$

Since 
$$\Gamma^{\beta}=\gamma\Gamma^{\lambda}+(1-\gamma)\Delta,~\bN^{\beta}=\gamma \bN^{\lambda}$$
and
$$\Omega_{Y}^\gamma=\gamma \Gamma^{\lambda}_Y+(1-\gamma)\Delta_Y, ~ \bL^{\gamma,Y}=\gamma \bN^{\lambda,Y},$$
by Lemma \ref{better} and Lemma \ref{toricadjunction}, the g-pair $(Y, \Gamma^{\beta}_Y,\bN^{\beta,Y})$ has better singularities than $(Y,\Omega_Y^{\gamma},\bL^{\gamma,Y})$, which implies that $(Z, \Gamma^{\beta}_Z,\bN^{\beta,Z})$ has better singularities than $(Z,\Omega_{Z}^\gamma,\bL^{\gamma,Z})$ by Lemma \ref{betterbase}. So by \eqref{tt} we have 
$$\mld(Z\ni z, \Gamma^{\beta}_Z,\bN^{\beta,Z})\geq \delta\big(t,\delta(s,\epsilon)\big).$$
\end{proof}

\begin{lem}\label{fano}
Let $F$ be a $\Q$-factorial complete toric variety given by $(N,\Sigma)$ with $\rho(F)=1$ and $\dim F=r\geq 2$. Then 

(1) its fan $\Sigma$ has exactly $r+1$ rays generated by primitive elements $v_i$, $i=1,\ldots,r+1$, and there exist positive integers $q_i$ such that $\sum_{i=1}^{r+1} q_iv_i=0$;

(2) let $E_i$, $i=1,\cdots,r+1$, be the prime divisor over $F$ corresponding to $-v_i$, then
$$a(E_i,F,0)=\frac{q_1+\cdots+\widehat{q_i}+\cdots +q_{r+1}}{q_i},$$
where the hat indicates that we omit that term;

(3) let $\pi:F'\to F$ be an extremal toric divisorial contraction with the exceptional divisor $E_i$ for some $i$, then there exists a toric contraction $g:F'\to G$ such that $\dim G=r-1$.
\end{lem}
\begin{rem}
Since $\rho(F')=2$, $F'$ has only two extremal rays. One corresponds to $F'\to F$ and the other corresponds to $F'\to G$.
\end{rem}
\begin{proof}

(1) The first assertion is showed in the proof of \cite[Lemma 3.3]{BC21}.

\iffalse
The lemma is proved in \cite{BC21} except the moreover part. For the convenience of readers, we will give a complete proof here.

Since $F$ be $\Q$-factorial and complete toric variety  with $\rho(F)=1$, the fan $\Sigma_F$ has exactly $r+1$ rays, say
$R_1,\dots, R_{r+1}$, and each cone generated by the subsets of 
$v_1,\dots,v_{r+1}$ of size $\le r$ is a cone in $\Sigma_F$.

Let $v_i$ be the primitive element of the ray $R_i$, $i=1,\ldots,r+1$, in the fan $F$. Then there exist integers $q_i$, such that $\sum_{i=1}^{r+1} q_iv_i=0$ and at least one $q_i$ are positive. We claim that all $q_i$ are positive integers. Indeed, if this is not the case, rearranging indices we can find $1\leq s\leq r$ such that $q_1,\cdots q_s$ are positive while $q_{s+1},\cdots q_{r+1}$ are not positive. Since $$q_1v_1+\cdots+q_sv_s=-q_{s+1}v_{s+1}\cdots -q_{r+1}v_{r+1},$$
the cone generated by $v_1,\dots,v_s$ has a non-trivial intersection with the cone generated by $v_{s+1},\dots,v_{r+1}$, which leads to a contradiction.  
\fi
\medskip
(2) Without loss of generality, we may suppose that $i=r+1$. Let $\Delta_F$ be the toric boundary divisor of $F$, then $a(D,F,\Delta_F)=0$ for any toric prime divisor $D$ over $F$, which implies that $a(D,F,0)$ is equal to the coefficient of $D$ in the pullback of $\Delta_F$ (denoted by $\mult_{D} \Delta_F$).

Since $-v_{r+1}$ is in the interior of the cone $\sigma$ generated by $v_1,\cdots,v_{r}$, the center of $E_{r+1}$ is contained in the affine chart $U_{\sigma}$. On the chart $U_{\sigma}$, $\Delta_F$ is determined by $m\in N^*\otimes \Q$ with $\langle m,v_i\rangle=1$ for $i=1,\cdots,r$. Then
$$\mult_{E_{r+1}} \Delta_F=\langle m,-v_{r+1}\rangle
%=\frac{\langle m,-q_{r+1} v_{r+1}\rangle}{q_{r+1}}
=\frac{\langle m,q_1v_1+\cdots +q_rv_r\rangle}{q_{r+1}}=\frac{q_1+\cdots +q_r}{q_{r+1}}.$$

(3) Without loss of generality, we may suppose that $i=r+1$. The toric variety $F'$ is given by $(N,\Sigma')$, where $\Sigma'$ is the star subdivision of $\Sigma$ along $-v_{r+1}$, more precisely, 
$$\Sigma'=(\Sigma\setminus \{\sigma\})\cup \Sigma^*(\sigma)$$
where $\sigma$ is the cone generated by $v_1,\cdots,v_r$ and $\Sigma^*(\sigma)$ is the set of all cones generated by subsets of $\{-v_{r+1},v_1,\cdots,v_r\}$ not containing $\{v_1,\cdots,v_r\}$.

Let $\phi:N\to N/(\Z v_{r+1}):=N_G$ be the quotient map and let
$$\Sigma_{G}=\{\phi_{\R}(\tau)\subset (N_G)_{\R}\mid \tau \in \Sigma \text{ and } -v_{r+1}\in \tau\}.$$
Then $\Sigma_G$ is a fan in $(N_{G})_{\R}$ (Exercise 3.2.7 in \cite{CLS11}).
Let $G$ be the toric variety given by $(N_G,\Sigma_G)$. We claim that $\phi$ is compatible with $\Sigma$ and $\Sigma_G$, that is, for any $\tau\in \Sigma$, $\phi_{\R}(\tau)$ is contained in some cone in $\Sigma_G$. Indeed, the claim holds obviously when $-v_{r+1}\in \tau$, so we may assume that $-v_{r+1}\notin \tau$. Then $\tau$ is generated by a subset $S$ of $\{v_1,\cdots,v_{r+1}\}$ not containing $\{v_1,\cdots,v_r\}$. Let $\tau'$ be the cone generated by $S'$ where
$$S' =
\begin{cases}
S\cup\{-v_{r+1}\},  & \text{if $v_{r+1}\notin S$,} \\
(S\setminus\{v_{r+1}\})\cup\{-v_{r+1}\}, & \text{if $v_{r+1}\in S$.}
\end{cases}
$$ 
Then $\phi_{\R}(\tau')\in \Sigma_G$ and $\phi_{\R}(\tau)=\phi_{\R}(\tau')$. Therefore the claim holds and then $\phi:N\to N_G$ determines a toric contraction from $F$ to $G$.
\end{proof}

\iffalse
\begin{lem}
Let $X,Z$ be two toric varieties given by $(N_X,\sum_X)$ and  $(N_Z,\sum_Z)$ and $f:X\to Z$ be a toric contraction given by a linear $F:N_X\to N_Y$. Suppose $v$ is a primitive element in $N_X$ such that $F(v)=0$ and both two ray generated by $v,-v$ are cones in  $\sum_X$. Then $f:X\rightarrow Z$ can be factored as $X\xrightarrow{g} Y \xrightarrow{h} Z$ where $h,g$ are toric contractions and $\dim Y=\dim X-1$.
\end{lem}
\begin{proof}
\end{proof}
\fi

\begin{lem}\label{diagram}
Let $f:X\rightarrow Z$ be a toric Mori fiber space of relative dimension $r\geq 2$ where $X$ is $\Q$-factorial. Then there is a commutative diagram 
\begin{equation*}
\begin{aligned}
\xymatrix{
	W \ar[r]^g \ar[d]^{\pi} & Y \ar[r]^h & Z\\
	X \ar[urr]_f
} 
\end{aligned}
\end{equation*}
such that
\begin{itemize}
\item $\pi,h,g$ are toric contractions,
\item $\pi:W\rightarrow X$ is an extremal toric divisorial contraction with the exceptional divisor $E$ satisfying $a(E,X,0)\leq r$, and
\item $\dim W-1=\dim Y>\dim Z$.
\end{itemize} 
\end{lem}
\begin{proof}
By Lemma \ref{contraction-over-torus}, over the torus $T_Z$ in $Z$, $f^{-1}(T_Z)$ is isomorphic to $F\times T_Z$, where $F$ is a general fiber of $f$. Since $f:X\to Z$ is a Mori fiber space, $F$ is a Fano variety with $\rho(F)=1$. Moreover, $F$ is $\Q$-factorial, as by Lemma \ref{contraction-over-torus} its fan $\Sigma_F$ is a sub-fan of the fan $\Sigma_X$ of $X$ which is simplicial. By Lemma \ref{fano} (1), the fan $\Sigma_F$ has exactly $r+1$ rays generated by primitive elements $v_i$, $i=1,\ldots,r+1$, and there exist positive integers $q_i$ such that $\sum_{i=1}^{r+1} q_iv_i=0$. Pick $e$ such that $q_e\geq q_i$ for any $i=1,\cdots,r+1$ and denote by $E_F$ the toric prime divisor over $F$ corresponding to $-v_e$. Extracting $E_F$ gives an extremal contraction $F'\to F$. By Lemma \ref{fano} (2), there is a toric contraction $F'\to G$ with $\dim G=r-1$. 

The closure $E$ of the exceptional divisor $E_F\times T_Z$ of $F'\times T_Z\to F\times T_Z$ is a toric divisor over $X$, so it determines an extremal toric divisorial contraction $\pi:W\rightarrow X$ with the exceptional divisor $E$. Then $\rho(W/Z)=2$. Over $T_Z$, the two contractions $W\to X$ and $F'\times T_Z\to F\times T_Z$ coincides. By Lemma \ref{fano} (2), we have
$$a(E,X,0)=a(E_F\times T_Z,F\times T_Z,0)\leq r.$$

Let $g:W\to Y$ be a $(-E)$-negative toric extremal contraction over $Z$. Then $\rho(Y/Z)=1$. Over $T_Z$, the restriction $g|_{T_Z}:F'\times T_Z\to Y|_{T_Z}$ is either an isomorphism or a $(-E_F\times T_Z)$-negative toric extremal contraction over $T_Z$. But the former case is impossible because $\rho(F')=2$ and $\rho(Y/Z)=1$. Note that $F'\times T_Z$ has only two extremal rays over $T_Z$. One corresponds to $F'\times T_Z\to F\times T_Z$ and the other corresponds to $F'\times T_Z\to G\times T_Z$. So $g|_{T_Z}$ coincides with one of them. It is impossible that $g|_{T_Z}$ coincides with $F'\times T_Z\to F\times T_Z$ because  $-E_F\times T_Z$ is ample over $F\times T_Z$. So $g|_{T_Z}$ coincides with $F'\times T_Z\to G\times T_Z$, which implies that $\dim Y=\dim W-1$.
%Since $-E_F\times T_Z$ is ample over $F\times T_Z$, the contraction $g|_{T_Z}:F'\times T_Z\to Y|_{T_Z}$ coincides with $F'\times T_Z\to G\times T_Z$, which implies that $\dim Y=\dim W-1$.
\end{proof}

\begin{lem}\label{MFS}
Assume that Theorem \ref{generalized} holds in relative dimension $\leq r-1$. Then Theorem \ref{generalized} holds in relative dimension $r$ when $f:X\rightarrow Z$ is a toric Mori fiber space and $X$ is $\Q$-factorial.
\end{lem}
\begin{proof}
By Lemma \ref{dim1}, we may suppose that the relative dimension $r\geq 2$. By taking a toric $\Q$-factorialisation, we can assume $X$ is $\Q$-factorial. By Lemma \ref{diagram}, there is a commutative diagram 
%as in \eqref{dia} 
$$\xymatrix{
	W \ar[r]^g \ar[d]^{\pi} & Y \ar[r]^h & Z\\
	X \ar[urr]_f
}$$ 
satisfying the properties listed in that lemma. Let $\Delta_W$ be the the toric boundary divisor of $W$, then $K_W+\Delta_W=\pi^*(K_X+\Delta)$. Write $K_W+B_W+\bM_W=\pi^*(K_X+B+\bM_X)$. 
Let
$$\Gamma^{\theta}_W=\theta B_W+(1-\theta) \Delta_W ~\text{ and }~ \bN^{\theta}=\theta \bM, \quad \text{ where } \theta=1/r.$$
Since $a(E,X,0)\leq r$, the coefficient of $E$ in $B_W$ is bounded below by $1-r$.  Then $\Gamma^{\theta}_W\geq 0$ since the coefficient of $E$ in $\Delta_W$ is 1. 

By construction, $\mld(W/Z\ni z,\Gamma^{\theta}_W, \bN^{\theta})\geq \frac{\epsilon}{r}$. Applying Lemma \ref{complex} to $(W,\Gamma^{\theta}_W, \bN^{\theta})$ over $Z$ (taking $s=1$ and $t=r-1$ in the lemma), we deduce that if we let 
$$\Omega_W^{\beta}=\beta \Gamma^{\theta}_W +(1-\beta)\Delta_W ~\text{ and }~ \bL^{\beta}=\beta \bN^{\theta}  \quad \text{ where } \beta=1/(r-1)!,$$
and $(Z,\Omega_Z^{\beta},\bL^{\beta,Z})/U$ be the g-pair given by adjunction for $h\circ g: (W,\Omega_W^{\beta},\bL^{\beta})\rightarrow Z$, then
$$\mld(Z\ni z,\Omega_Z^{\beta},\bL^{\beta,Z})\geq \frac{(\epsilon/r)^{2^r}}{2^{2^{r}-1}\cdot\prod\limits_{i=1}^{r-1} i^{2^i}}=\delta(r,\epsilon).$$
It is easy check that 
$$\Omega_W^{\beta}=\alpha B_W +(1-\alpha)\Delta_W ~\text{ and }~ \bL^{\beta}=\alpha \bM  \quad \text{ where } \alpha=\theta\beta=1/r!.$$
Hence
$$K_W+\Omega_W^{\beta}+\bL^{\beta}_W=\pi^*(K_X+\Gamma^{\alpha}+\bN^{\alpha}_X),$$
where
$$\Gamma^{\alpha}=\alpha B+(1-\alpha)\Delta ~ \text{ and } ~\bN^{\alpha}=\alpha \bM.$$
Therefore $(Z,\Omega_Z^{\beta},\bL^{\beta,Z})/U$ is also the g-pair given by adjunction for $f:(X,\Gamma^{\alpha},\bN^{\alpha})\rightarrow Z$.  This finishes the proof of the lemma.
\end{proof}

\begin{proof}[Proof of Theorem \ref{generalized}]
By induction on relative dimension we may assume that the theorem holds in relative dimension $\leq r-1$. Taking a toric $\Q$-factorization of $X$ and running an MMP on $K_X$ over $Z$, we may assume that $X$ is $\Q$-factorial and it has a toric Mori fiber space structure $X\rightarrow Y/Z$. 

If $Y\rightarrow Z$ is birational, we can replace $Z$ by $Y$, then we are done by Lemma \ref{MFS}. Otherwise $\dim Y>\dim Z$. Denote $s=\dim X-\dim Y$ and $t=\dim Y-\dim Z$, then $r=s+t$. Applying Lemma \ref{complex}, we deduce that if we let
$$
  \Gamma^{\beta}=\beta B+(1-\beta)\Delta ~ \text{ and } ~\bN^{\beta}=\beta \bM \quad \text{ where } \beta=1/(s!t!)
$$
and $(Z, \Gamma^{\beta}_Z,\bN^{\beta,Z})/U$ be the g-pair given by adjunction for $f:(X,\Gamma^{\beta},\bN^{\beta})\rightarrow Z$, then 
\begin{align}\label{q}
\mld(Z\ni z, \Gamma^{\beta}_Z,\bN^{\beta,Z}) \geq \frac{\epsilon^{2^{r}}}{2^{2^{r}-1}\prod\limits_{i=1}^s i^{2^{i+t}}\cdot\prod\limits_{i=1}^t i^{2^i}}.
\end{align}

Let
  $$
  \Gamma^{\alpha}=\alpha B+(1-\alpha)\Delta ~ \text{ and } ~\bN^{\alpha}=\alpha \bM \quad \text{ where } \alpha=1/r!.
  $$
Then we have
 $$
  \Gamma^{\alpha}=\theta \Gamma^{\beta}+(1-\theta)\Delta ~ \text{ and } ~\bN^{\alpha}=\theta \bN^{\beta} \quad \text{ where } \theta=\alpha/\beta=(s!t!)/r!.
 $$
Let $(Z, \Gamma^{\alpha}_Z,\bN^{\alpha,Z})$ be the g-pair given by adjunction for $f:(X,\Gamma^{\alpha},\bN^{\alpha})\rightarrow Z$. By Lemma \ref{better} and Lemma \ref{toricadjunction},  $(Z, \Gamma^{\alpha}_Z,\bN^{\alpha,Z})$ has better singularities than
$$(Z, \theta \Gamma^{\beta}_Z +(1-\theta) \Delta_Z,\theta\bN^{\beta,Z}).$$
Hence by \eqref{q} we have
\begin{align*}
\mld(Z\ni z, \Gamma^{\alpha}_Z,\bN^{\alpha,Z})&\geq \frac{s!t!}{r!}\cdot\frac{\epsilon^{2^{r}}}{2^{2^{r}-1}\prod\limits_{i=1}^s i^{2^{i+t}}\cdot\prod\limits_{i=1}^t i^{2^i}}\\
& \geq \frac{\epsilon^{2^r}}{2^{2^r-1}\prod\limits_{i=1}^r i^{2^i}}=\delta(r,\epsilon).
\end{align*}
\end{proof}

\begin{proof}[Proof of Theorem \ref{thmtech}]
It is a special case of Theorem \ref{generalized}.
\end{proof}

\begin{proof}[Proof of Theorem \ref{thm3}]
By Lemma \ref{lalala}, shrinking $Z$ around $z$, we may suppose that $(X,B)$ is lc. Since $(X,B)$ is a toric lc pair, we have $B\leq \Delta$, where $\Delta$ is the toric boundary divisor of $X$.
Let
  $$
  \Gamma^{\alpha}=\alpha B^++(1-\alpha)\Delta \quad \text{ where } \alpha=1/r!.
  $$
Then $\Gamma^{\alpha}\geq B$. Let $(Z, \Gamma^{\alpha}_Z,\bN^{\alpha,Z})$ be the g-pair given by adjunction for $f:(X,\Gamma^{\alpha})\rightarrow Z$. By Theorem \ref{thmtech}, $\mld(Z\ni z, \Gamma^{\alpha}_Z,\bN^{\alpha,Z})\geq \delta=\delta(r,\epsilon)$. Denote the prime divisor $\overline{z}$ by $D$. Then the coefficient of $D$ in $\Gamma_Z^{\alpha}$ is bounded from above by $1-\delta$. This means that  $(X,\Gamma^{\alpha}+\delta f^*D)$ is lc over the generic point of $D$. Since $\Gamma^{\alpha}\geq B$, we deduce that $(X,B+\delta f^*D)$ is lc over the generic point of $D$.
\end{proof}

\begin{proof}[Proof of Theorem \ref{thm2}]
Let
  $$
  \Gamma^{\alpha}=\alpha B+(1-\alpha)\Delta \quad \text{ where } \alpha=1/r!
  $$
and let $(Z, \Gamma^{\alpha}_Z,\bN^{\alpha,Z})$ be the g-pair given by adjunction for $f:(X,\Gamma^{\alpha})\rightarrow Z$. By Theorem \ref{thmtech}, $\mld(Z\ni z, \Gamma^{\alpha}_Z,\bN^{\alpha,Z})\geq \delta(r,\epsilon)$.  Hence $\mld(Z\ni z, 0)\geq \delta(r,\epsilon)$ and the first assertion holds.

The second assertion is an immediate consequence of Theorem \ref{thm3} (taking $B=0$ in Theorem \ref{thm3}).
\end{proof}

\begin{proof}[Proof of Theorem \ref{thm1}]
This is a direct consequence of Theorem \ref{thm2}.
%applying Theorem \ref{thm2} to all  codimension $\geq 1$ points of $Z$.
\end{proof}

%%%%%%%%%%%%%%%%%%%%%%%%%%%%%%%%%%%%%

\end{document}